\documentclass{amsart}

\usepackage[nohug]{diagrams}
\usepackage[usenames,dvipsnames]{pstricks}
\usepackage{epsfig}
\usepackage{pst-grad} 
\usepackage{pst-plot} 

\newtheorem{theorem}{Theorem}[section]
\newtheorem{proposition}[theorem]{Proposition}
\newtheorem{lemma}[theorem]{Lemma}

\theoremstyle{definition}

\newtheorem{example}[theorem]{Example}

\newcommand{\id}{\text{id}}
\newcommand{\s}{\text{S}}
\newcommand{\dist}{\text{d}}
\numberwithin{equation}{section}

\begin{document}

\title{A coarse invariant}

\author{A. Fox}
\address{Saint Francis University, Loretto, PA 15940}
\email{adfst5@francis.edu}

\author{B. LaBuz}
\address{Saint Francis University, Loretto, PA 15940}
\email{blabuz@@francis.edu}

\author{R. Laskowsky}
\address{Saint Francis University, Loretto, PA 15940}
\email{ralst8@@francis.edu}

\subjclass[2000]{}
\date{}

\begin{abstract}
This note extends the invariant of metric spaces under bornologous equivalences defined in \cite{MMS} to the coarse category.
\end{abstract}

\maketitle
\tableofcontents

\section{Introduction}

A coarse function $f:X\to Y$ between metric spaces is a function that is bornologous and proper. $f$ is bornologous if for each $N>0$ there is an $M>0$ such that if $d(x,y)\leq N$, $d(f(x),f(y))\leq M$. In this setting we call $f$ proper if inverse images of bounded sets are bounded.

Notice bornology is dual to continuity. Thus bornology is a fundamental concept of coarse (or large scale) geometry just as continuity is a fundamental concept of topology (small scale geometry). We are studying the large scale behavior of functions and large scale properties of spaces.

Two metric spaces $X$ and $Y$ are coarsely equivalent if there are coarse functions $f:X\to Y$ and $g:Y\to X$ such that $g\circ f$ is close to $\id_X$ and $f\circ g$ is close to $\id_Y$. Two functions $f_1$ and $f_2$ are close if $d(f_1(x),f_2(x))$ is uniformly bounded. A standard reference for the preceding concepts and coarse geometry in general is \cite{Roe}.

In \cite{MMS} an invariant in the bornologous category is constructed. This note extends the construction in \cite{MMS} to the coarse category. Bornologous equivalence is more strict than coarse equivalence. For bornologous equivalence $f\circ g$ and $g\circ f$ are required to be the identity on the nose. Coarse equivalence can be viewed as being in the category where, instead of considering functions, one considers equivalence classes of functions. Two functions are equivalent if they are close.

The standard example of two coarsely equivalent spaces is $\mathbb R$ and $\mathbb Z$ (see Example \ref{integers}). Of course these spaces cannot be bornologously equivalent because they do not have the same cardinality. We can explain interest in the coarse category as opposed to the bornologous category as follows. Since we are interested in large scale behavior, we should ignore all small scale behavior including cardinality. We should not care whether the number of points in a neighborhood is finite or infinite.

\section{Previous construction}

We recall the construction from \cite{MMS}. Fix a basepoint $x_0\in X$. Given $N>0$, an $N$-sequence in $X$ based at $x_0$ is an infinite list $x_0,x_1,\ldots$ of points in $X$ with $d(x_i,x_{i+1})\leq N$ for each $i\geq 0$. Since we are interested in the large scale structure of $X$, we are only interested in sequences that go to infinity. An $N$-sequence $x_0,x_1,\ldots$ goes to infinity if $d(x_0,x_i)\to\infty$. Let $\s_N(X,x_0)$ be the set of all $N$-sequences in $X$ based at $x_0$ that go to infinity.

We call two sequences $s,t\in \s_N(X,x_0)$ equivalent if there is a finite list $s_0,
\ldots,s_n\in \s_N(X,x_0)$ with $s_0=s$, $s_n=t$, and for each $i\geq 0$, $s_{i+1}$ is either a subsequence of $s_i$ or $s_i$ is a subsequence of $s_{i+1}$. If $s_i$ is a subsequence of $s_{i+1}$ we say $s_{i+1}$ is a supersequence of $s_i$. Let $[s]_N$ denote the equivalence class of $s$ in $\s_N(X,x_0)$ and let $\sigma_N(X,x_0)$ be the set of equivalence classes.

The cardinality of the set $\sigma_N(X,x_0)$ is the desired invariant. It essentially determines the number of different ways of going to infinity in $X$. Since this cardinality depends on $N$, we have the following definition. For each integer $N>0$ there is a function $\phi_N:\sigma_N(X,x_0)\to\sigma_{N+1}(X,x_0)$ that sends the equivalence class $[s]_N$ to the equivalence class $[s]_{N+1}$. $X$ is said to be $\sigma$-stable if there is a $K>0$ for which $\sigma_N$ is a bijection for each integer $N\geq K$. If $X$ is $\sigma$-stable let $\sigma(X,x_0)$ denote the cardinality of $\sigma_K(X,x_0)$.

It would be better to call $X$ ``$\sigma$-stable with respect to $x_0$'' since apparently this definition depends on basepoint. In fact it does not; this issue is addressed in the next section.

The following is the main theorem of \cite{MMS}. It is the theorem that we wish to extend to coarse equivalences.

\begin{theorem}[\cite{MMS}, Theorem 3.2]
Suppose $f:X\to Y$ is a bornologous equivalence between metric spaces. Let $x_0$ be a basepoint of $X$ and set $y_0=f(x_0)$. Suppose $X$ and $Y$ are $\sigma$-stable. Then $\sigma(X,x_0)=\sigma(Y,y_0)$.
\end{theorem}

\section{Change of basepoint in $\sigma$-stable spaces}

As mentioned above, the definition of $\sigma$-stable depends on the choice of basepoint. We show that in fact a space being $\sigma$-stable is independent of basepoint.

\begin{lemma}\label{zlemma}
Suppose $x_0,y_0\in X$ and $n\geq \dist(x_0,y_0)$. Let $z_n:\sigma_n(X,x_0)\to\sigma_n(X,x_1)$ be the function that sends the equivalence class of a sequence $x_0,x_1,x_2,\ldots$ to the equivalence class of $y_0,x_0,x_1,x_2,\ldots$. Then $z_n$ is a bijection.
\end{lemma}

\begin{proof}
Let $w_n:\sigma_n(X,y_0)\to \sigma_n(X,x_0)$ be the function that sends the equivalence class of a sequence $y_0,y_1,y_2,\ldots$ to the equivalence class of $x_0,y_0,y_1,y_2,\ldots$. We show that $z_n$ and $w_n$ compose to form the identities and thus $z_n$ must be a bijection. Suppose $[(x_i)]\in\sigma_n(X,x_0)$. Then $(w_n\circ z_n)([(x_i)])$ is the equivalence class of the sequence $x_0,y_0,x_0,x_1,\ldots$ which is a supersequence of $(x_0)$. Similarly, $z_n\circ w_n$ is the identity on $\sigma_n(X,y_0)$.
\end{proof}

\begin{proposition}
Suppose a metric space $X$ is $\sigma$-stable with respect to a basepoint $x_0\in X$. Let $y_0\in X$. Then $X$ is $\sigma$-stable with respect to $y_0$ and $\sigma(X,x_0)=\sigma(X,y_0)$.
\end{proposition}

\begin{proof}
Let $N\in\mathbb N$ be such that $\phi_n:\sigma_n(X,x_0)\to\sigma_{n+1}(X,x_0)$ is a bijection for all $n\geq N$. Choose $M\in\mathbb N$ such that $M\geq N,\dist(x_0,x_1)$. Suppose $n\geq M$.  Then the following diagram commutes.

\begin{diagram}
\sigma_{n+1}(X,x_0) & \rTo^{z_{n+1}} & \sigma_{n+1}(X,y_0)\\
\uTo^{\phi_n}       &                & \uTo_{\psi_n}\\
\sigma_n(X,x_0)     & \rTo^{z_n}     & \sigma_n(X,y_0)\\
\end{diagram}

Since $\phi_n$, $z_n$, and $z_{n+1}$ are bijections, so is $\psi_n$.
\end{proof}

\endproof

\section{The invariant}

\begin{theorem}
Suppose $X$ and $Y$ are coarsely equivalent and $\sigma$-stable. Then $\sigma(X)=\sigma(Y)$
\end{theorem}

\proof
Suppose $f:X\to Y$ and $g:Y\to X$ compose a coarse equivalence. Let $x_0$ be a basepoint in $X$ and $y_0 = f(x_0)$ be a basepoint in $Y$.
Because $g\circ f$ is close to $\id_X$, we can say that there is a $D$ such that $\dist(x,g\circ f(x))\leq D$ for all $x \in X$.  Let $K$ be the integer provided by $X$ being $\sigma$-stable and $K'$ be the integer provided by $Y$ being $\sigma$-stable.  As $f$ is bornologous, there is an $M$ such that $f$ sends $K$-sequences to $M$-sequences in $Y$.  We can assume $ M\geq K'$. Similarly, because $g$ is bornologous, there is an $L$ such that $g$ sends $M$-sequences to $L$-sequences in $X$, choosing $L \geq D$.

Let $z_L:\sigma_L(X, x_0)\to \sigma_L(X, g\circ f(x_0))$ be the function that sends the equivalence class of $x_0, x_1, \ldots$ to the equivalence class of $g\circ f(x_0),x_0, x_1, \ldots$.  We chose $L\geq D$, so we can say that this addition does not prevent $x_0, g\circ f(x_0), x_1, \ldots$ from being an $L$ sequence. By \ref{zlemma} we know $z_L$ is a bijection.

Let $f_K$ be the function that sends an element$[s]_K\in\sigma_K(X,x_0)$ to the element $[f(s)]_M\in\sigma_M(Y,f(x_0))$ and let $g_M$ be the function that sends an element $[s]_M\in\sigma_M(Y,f(x_0))$ to the element $[g(s)]_L\in\sigma(X,g\circ f(x_0))$.

We show the following diagram commutes:

\begin{diagram}
\sigma_L(X, g\circ f(x_0)) &   &                  &   &                     \\
\uTo^{z_L}                 &\luTo(4,4)^{g_M}  &                  &   &                     \\
\sigma_L(X, x_0)           &   &                  &   &                     \\
\uTo^{\phi_{KL}}           &   &                  &   &                     \\
\sigma_K(X, x_0)           &   &    \rTo_{f_K}    &   & \sigma_M(Y, f(x_0)) \\
\end{diagram}

Let $(x_n)$ be a $K$-sequence in $X$. Then $g\circ f([(x_n)])$ is the equivalence class of the sequence $g\circ f(x_0), g\circ f(x_1), \ldots$ and $z_L\circ\phi_{KL}([x_n])$ is the equivalence class of the sequence $g\circ f(x_0),x_0,x_1,\ldots$. Consider the sequence \[g\circ f(x_0), x_0, x_1, g\circ f(x_1),g\circ f(x_2), x_2, x_3, g\circ f(x_3),g\circ f(x_4), \ldots\] There are three distances to consider: the distance between successive elements of ${x_n}$, the distance between successive elements of ${g\circ f(x_n)}$, and the distance between any $x_i$, and its counterpart $g\circ f(x_i)$.
Because $d(x_i, g\circ f(x_i)) \leq D$, $d(x_i, x_{i+1}) \leq K$, and $d(g\circ f(x_i), g\circ f(x_{i+1})) \leq L$, the unioned sequence is an $L$-sequence.  Further, because the two sequences ${x_n}$ and ${g\circ f(x_n)}$ are visited in order, we can say that ${x_n}$ and ${g\circ f(x_n)}$ are both subsequences of this union.  Thus, the diagram commutes.

Since $z_L\circ\phi_{KL}$ is a bijection, $f_K$ must be one-to-one.

Symmetrically we can see that the following diagram commutes where $S$ is chosen so that $\dist(y,f\circ g(y))\leq S$ for all $y\in Y$ and $\dist(f(x),f(y))\leq S$ whenever $\dist(x,y)\leq L$.

\begin{diagram}
                           &   &                  &   &  \sigma_S(Y, f\circ g\circ f(x_0))  \\
                           &   &                  &  \ruTo(4,4)^{f_L}   &  \uTo^{z_M}                         \\
                           &   &     &   &  \sigma_S(Y, f(x_0))                \\
                           &   &                  &   &  \uTo^{\psi_{MS}}                   \\
\sigma_L(X, g\circ f(x_0)  &   &    \lTo_{g_M}    &   &  \sigma_M(Y, f(x_0))                \\
\end{diagram}

Thus $g_M$ must be one-to-one which forces $f_K$ to be onto. Then we have that $f_K$ is a bijection.

\endproof

\section{Some examples}

We begin with the standard example of a coarse equivalence. 

\begin{example}\label{integers}\cite{Roe}
Consider $\mathbb R$ and $\mathbb Z$ as metric spaces under the usual metric. Let $f:\mathbb R\to \mathbb Z$ be the floor function, $x\mapsto \lfloor x\rfloor$. Let $g:\mathbb Z\to \mathbb R$ be the inclusion, $n\mapsto n$. It is easy to see that $f$ and $g$ are coarse and that $g\circ f$ and $f\circ g$ are close to the identities ($g\circ f$ is the identity). Corollary 3.7 in \cite{MMS} says that $\sigma(\mathbb R)=2$. Since $\mathbb Z$ is coarsely equivalent to $\mathbb R$ we must have $\sigma(\mathbb Z)=2$ also. Of course we can see these two sequences in $\mathbb Z$. 

\end{example}

Next we give another way to calculate $\sigma(V)$ where $V$ is the vase from \cite[Example 1.3]{MMS}. We first give a basic lemma.

\begin{lemma}
Suppose $f:X\to Y$ is any function and $g:Y\to X$ is bornologous. Suppose that $g\circ f$ is close to the identity on $X$. Then $f$ is proper.
\end{lemma}

\begin{proof}
Suppose $A\subset Y$ is bounded, say $\dist(x,y)\leq N$ for all $x,y\in A$. Suppose $x,y\in f^{-1}(A)$. Then $f(x),f(y)\in A$ so $\dist(f(x),f(y))\leq N$. Since $g$ is bornologous there is an $M>0$ so that $\dist(g\circ f(x),g\circ f(y))\leq M$. Now since $g\circ f$ is close to the identity there is an $R>0$ so that $\dist(g\circ f(x),x),\dist(g\circ f(y),y)\leq R$. Thus $\dist(x,y)\leq \dist(x,g\circ f(x))+\dist(g\circ f(x),g\circ f(y))+\dist(g\circ f(y),y)\leq M+2R$.
\end{proof}

\begin{example}
Let $V=\{(-1,y):y\geq 1\}\cup \{(x,1):-1\leq x\leq 1\}\cup \{(1,y):y\geq 1\}\subset \mathbb R^2$. Following \cite{MMS} we will use the taxicab metric which is bornologously (and therefore coarsely) equivalent to the standard metric. We show that $V$ is coarsely equivalent to the ray $[1,\infty)$ and therefore $\sigma(V)=\sigma[1,\infty)$.

First let us note that $\sigma[1,\infty)$ is fairly easy to calculate. Suppose $N\geq 1$. By \cite[Lemma 2.4]{MMS} an $N$-sequence $s=\{s_i\}$ in $[1,\infty)$ goes to infinity if and only if $\lim_{i\to\infty}s=\infty$. In particular $[(i)]\in\sigma_N[1,\infty)$. Also, given $[s]\in\sigma_N[1,\infty)$, since $\lim_{n\to\infty}s=\infty$ we see that $s$ is equivalent to an increasing sequence $t$ and if we put $t$ and $(i)$ together using the order on $\mathbb R$ we obtain the desired equivalence between $s$ and $(i)$. We have shown that $\sigma_N[1,\infty)=\{[(i)]\}$ so $\sigma[1,\infty)=1$.

Now we define a coarse equivalence between $V$ and $[1,\infty)$. Define $f:V\to [1,\infty)$ to send a point $(x,y)\in V$ to $y\in[0,1)$. Let $g:[1,\infty)\to V$ send $y\in [1,\infty)$ to $(1,y)\in V$. We have that $f\circ g(y)=y$ for all $y\in [1,\infty)$ so $f\circ g$ is the identity on the nose. Given $(x,y)\in V$, $\dist(g\circ f(x,y),(x,y))=|x-1|+|y-y|\leq 2$ so $g\circ f$ is close to the identity.

By the lemma we need only to check that $f$ and $g$ are bornologous. Suppose $N>0$, $(x_1,y_1),(x_2,y_2)\in V$, and $\dist((x_1,y_1),(x_2,y_2))\leq N$. Then $\dist(f(x_1,y_1),f(x_2,y_2))=|y_1-y_2|=\dist((x_1,y_1),(x_2,y_2))-|x_1-x_2|\leq N-|x_1-x_2|\leq N$. Now suppose $y_1,y_2\in [1,\infty)$ and $|y_1-y_2|\leq N$. Then $\dist(g(y_1),g(y_2))=\dist((1,y_1),(1,y_2))=|1-1|+|y_1-y_2|\leq N$.
\end{example}

\begin{proposition}
Let $V$ be the vase from the previous example. Then $V$ is not coarsely equivalent to $\mathbb R$.
\end{proposition}

\begin{proof}
According to the previous example $\sigma(V)=1$ and according to \cite[Corollary 3.7]{MMS} $\sigma(\mathbb R)=2$.
\end{proof}


\begin{thebibliography}{99}

\bibitem{MMS} B. Miller, J. Moore, and L. Stibich. \emph{An invariant of metric spaces under bornologous equivalences.} Mathematics Exchange. 7 (2010) 12--19.


\bibitem{Roe}
J. Roe. \emph{Lectures on coarse geometry, University lecture series, 31.} American Mathematical Society, Providence, RI, 2003.

\end{thebibliography}
\end{document}